\documentclass[12pt]{amsart}
\usepackage{tikz-cd}

\newtheorem{thm}     {Theorem}[section]
\newtheorem{prop}    [thm]{Proposition}

\newtheorem{cor}     [thm]{Corollary}
\newtheorem{lemma}   [thm]{Lemma}
\newtheorem{remark}   [thm]{Remark}

\newcommand{\g}{\mathfrak g}

\newcommand{\B}{\mathbb B}
\newcommand{\C}{\mathbb C}
\newcommand{\R}{\mathbb R}
\newcommand{\T}{\mathbb T}
\newcommand{\Z}{\mathbb Z}
\newcommand{\id}{{\rm id}}
\newcommand{\Span}{{\rm Span}}

\def\Re{{\rm Re\,}}
\def\Im{{\rm Im\,}}

\def\<{\langle}
\def\>{\rangle}
\def\({\left(}
\def\){\right)}

\def\Aut{{\rm Aut}}
\def\Sym{{\rm Sym}}

\def\Mat{{\rm Mat}}
\def\Diag{{\rm Diag}}
\def\Int{{\rm Int}}
\def\Ad{{\rm Ad}}

\def\alert{}

\begin{document}


\title[Realizing Lie groups as automorphisms
of bounded domains]
{Realizing semisimple Lie groups as holomorphic \\
automorphism groups of bounded domains}

\author{George Shabat}
\address{Russian State University for the Humanities,
Moscow, 125267, Russia}
\email{george.shabat@gmail.com}

\author{Alexander Tumanov}
\address{Department of Mathematics, University of Illinois,
1409 West Green St., Urbana, IL 61801, U.S.A.}
\email{tumanov@illinois.edu}

\maketitle

\begin{abstract}
We consider a problem whether a given Lie group can be realized as the group of all biholomorphic automorphisms of a bounded domain in $\C^n$.
In an earlier paper of 1990, we proved the result for connected \emph{linear} Lie groups.
In this paper we prove the result for
a fairly large class of Lie groups including all connected real semisimple groups.
This class contains many non-linear Lie groups.

\end{abstract}

\section{Introduction}

Let $D\subset\C^n$ be a bounded domain. H. Cartan \cite{Cartan} proved that the group $\Aut(D)$ of all biholomorphic automorphisms of $D$ is a (real finite dimensional) Lie group.
Is the converse true? In other words, which Lie groups can be realized as $\Aut(D)$ for a bounded domain $D\subset\C^n$?

Bedford and Dadok \cite{Bedford} and Saerens and Zame \cite{Saerens} proved that every \emph{compact} Lie group can be realized as $\Aut(D)$ for a bounded strongly pseudoconvex domain $D\subset\C^n$.
On the other hand,
Wong \cite{Wong} and Rosay \cite{Rosay} proved that if $D\subset\C^n$ is bounded, strongly pseudoconvex, and $\Aut(D)$ is \emph{not compact}, then $D$ is biholomorphically equivalent to the unit ball $\B^n\subset\C^n$.
Therefore, if the group is not compact, we cannot expect to realize it as $\Aut(D)$ for a bounded strongly pseudoconvex domain $D\subset\C^n$.

A Lie group is called \emph{linear} if it is isomorphic to a subgroup of a general linear group $GL(n,\R)$ of all real
nonsingular $n\times n$ matrices.
% We recall that a \emph{connected} linear Lie group is isomorphic to a \emph{closed} subgroup of $GL(n,\R)$.

We call a domain $D\subset\C^n$ a domain of \emph{bounded type} if $D$ is biholomorphically equivalent to a bounded domain.

In an earlier paper \cite{Shabat}, we proved that every (possibly non-compact) connected \alert{linear} Lie group can be realized as $\Aut(D)$, where $D\subset\C^n$ is a strongly pseudoconvex domain of bounded type.
Winkelmann \cite{Winkelmann} and Kan \cite{Kan} proved that every connected (possibly non-linear) Lie group can be realized as $\Aut(D)$, where $D$ is a \emph{complete hyperbolic Stein manifold}.

The question whether $D$ can be chosen a \emph{bounded domain} in $\C^n$ remains largely open.

In \cite{Shabat2024}, we answer the question in the affirmative for some examples of non-linear Lie groups including $\widetilde{SL}(2,\R)$, the universal cover of $SL(2,\R)$.
In this paper we extend the result to a fairly large class of Lie groups (Theorem \ref{Main-Algebraic})
including all connected real semisimple groups (Corollary \ref{Main-Semisimple}).

The paper is structured as follows.
In Section 2, we give a general plan for realizing non-linear Lie groups as holomorphic automorphisms groups. Loosely speaking, we show that if a group $G$ is a covering group of a realizable Lie group and the covering map can be obtained as a pullback of an analytic covering map with compact base, then $G$ can be realized.
In Section 3, we show that every linear Lie group $G$ acts by holomorphic transformations on a domain of bounded type. Moreover, this action is free, proper, it has totally real orbits, and this action reduces to the action of $G$ on its complexification.
In Section 4, we study covering groups of compact Lie groups.
In Section 5, we state and prove the main results. We use polar decomposition of matrices to construct a mapping from a linear Lie group to a compact Lie group needed for pulling back a covering map.
In Sections 6--8, we give technical details justifying the use of polar decomposition in Section 5.

\section{General results}
In this section, we closely follow \cite{Shabat2024}.

Recall that a \emph{group action} $G:X$ of a group $G$ on a set $X$ is a mapping $G\times X \to X$, which we denote as $(g,x)\mapsto gx$ or $(g,x)\mapsto g\cdot x$,
such that
$e x=x$ and $g_1 (g_2 x)=(g_1g_2) x$.
Here $e\in G$ is the identity.

A group action $G:X$ is \emph{free} (or with no fixed points) if for every $x\in X$, the map $G\to X$, $g\mapsto g x$ is injective.

A group action $G:X$ is \emph{proper} if the mapping $G\times X \to X\times X$, $(g,x)\mapsto (g x,x)$, is proper.
Here $G$ is a topological group, $X$ is a topological space, and the action $G\times X \to X$ is continuous.

A group action $G:X$ is \emph{holomorphic} if for every $g\in G$, the map $x\mapsto gx$ is biholomorphic. Here $X$ is a complex manifold.

Recall that a real submanifold $M$ of a complex manifold $X$
is \emph{totally real} if for every $z\in X$
the tangent space $T_zM$ does not contain complex lines.

\begin{prop}\cite{Bedford, Saerens, Shabat, Winkelmann}
\label{Common}
Let $G:\Omega$ be a holomorphic group action of a connected Lie group $G$ on a domain $\Omega\subset\C^n$, $n\ge 2$. Suppose the action is proper, free, and the orbits are totally real. Then a generic smooth small tubular $G$-invariant neighborhood $D$ of each orbit is strongly pseudoconvex, and $\Aut(D)$ is isomorphic to $G$.
\end{prop}

The proof consists of two steps. In the first step, one proves that every $f\in \Aut(D)$ extends smoothly to the most of the boundary $bD$ of the domain $D$.
In the second step, using local invariants of CR structure of $bD$ \cite{Chern-Moser}, by small perturbations one can rule out automorphisms other than the ones induced by the action of $G$.

If $G$ is compact, then $D$ is a bounded strongly pseudoconvex domain, and the existence of smooth extension follows by Fefferman's theorem.
In the case of non-compact $G$, our short paper \cite{Shabat} did not include full details of the first step. The proof can be found in \cite{Winkelmann}.

Let $G$ be a Lie group that can be realized using Proposition \ref{Common}. We describe a situation in which a covering group of $G$ also can be realized.

We first recall a general construction.
Let $H, H_1, H_2$ be sets (groups or manifolds), and let
$\phi_\nu:H_\nu\to H$ be surjections (resp. homomorphisms or smooth mappings). We will use notation
$$
H_1\times_H H_2=\{(h_1,h_2)\in H_1\times H_2: \phi_1(h_1)=\phi_2(h_2)\}.
$$
We will write $H_1\times_H H_2$ without mentioning the maps $\phi_\nu$ when they are clear from the context.
Breaking the symmetry, we sometimes call the projection $H_1\times_H H_2\to H_1$ the \emph{pullback} of $\phi_2$ by $\phi_1$.

We now introduce an auxiliary definition to lighten up the statement of the next result.
Let $X$ and $\widehat X$ be (locally closed) complex submanifolds in $\C^k$ and $\C^m$ respectively.
Let $f:\widehat X\to X$ be a holomorphic covering map.
Suppose there is $\epsilon>0$ such that for every $x\in X$,
all the elements of the preimage $f^{-1}(x)$ are at least $\epsilon$ units apart.
Suppose in addition that the $\epsilon$-thickening of $\widehat{X}$ in $\C^m$ is a domain of bounded type.
Then we call $f:\widehat X\to X$ an \emph{admissible}
covering map. For definiteness, by dilating $\widehat X$,
we can put, say $\epsilon=1$.

\begin{prop}\label{Cover}
Let $G$ be a connected Lie group.
Let $G:\Omega$ be a holomorphic free proper action with totally real orbits on a domain $\Omega\subset\C^n$ of bounded type.
Let $M\subset\Omega$ be the $G$-orbit of a point $z_0\in\Omega$.
Let $\psi:\Omega\to X$ be a holomorphic mapping onto a
(locally closed) complex submanifold $X\subset\C^k$.
Suppose $H:=\psi(M)\subset X$ is compact,
and $X$ is contained in a small neighborhood of $H$.
Let $f:\widehat{X}\to X$ be an admissible covering map, here $\widehat{X}\subset\C^m$.
Let $\widehat H=f^{-1}(H)$, and
let $\widehat{G}=G\times_H\widehat{H}$ be a covering group for $G$ obtained by pulling back $f$ by the mapping
$G\ni g\mapsto\psi(gz_0)\in H$. Then there is a strongly pseudoconvex domain $D\subset\C^{n+m}$ of bounded type with $\Aut(D)\simeq\widehat{G}$.
\end{prop}

\begin{proof}
We obtain a covering manifold $\widehat{\Omega}$ of $\Omega$ by pulling back $f$ by $\psi$, that is,
$$
\widehat{\Omega}=\Omega\times_X\widehat{X}=  \{(z,w)\in\Omega\times\widehat{X}:
\psi(z)=f(w) \}\subset\C^{n+m}.
$$
The action $G:\Omega$ lifts to an action
$\widehat G:\widehat\Omega$.
Indeed, let $\widehat g\in \widehat G$ be represented by a curve
$\widehat g:[0,1]\to G$ with
$\widehat g(0)=e$, $\widehat g(1)=g$.
Then for $(z,w)\in\widehat{\Omega}$, we define
\begin{equation*}
\widehat g\cdot(z,w)=(gz, \gamma(1)),
\end{equation*}
here
$\gamma:[0,1]\to\widehat{X}$ is such a continuous curve that
$\gamma(0)=w$ and $\psi(\widehat g(t)z)=f(\gamma(t))$
for all $0\le t \le 1$.

Since $f:\widehat{X}\to X$ is admissible, we can assume that for every $w_1\ne w_2\in\widehat{X}$, if $f(w_1)=f(w_2)$, then $|w_1-w_2|>2$.
We define
$$
\widehat{\widehat\Omega}
=\{(z,w)\in\Omega\times\C^m: \exists (z,w_0)\in\widehat\Omega, |w-w_0|<1 \}\subset\C^{n+m}.
$$
The action $\widehat G:\widehat \Omega$ extends to
$\widehat G:\widehat{\widehat\Omega}$, so it is locally independent of the $w$-component.
That is, if $(z,w_0)\in\widehat{\Omega}$,
$\widehat g\cdot(z,w_0)=(g z,w_1)$,
and $|w-w_0|<1$, then we have
$$
\widehat g\cdot(z,w)=(g z,w-w_0+w_1).
$$
Since $f$ is admissible, the domain $\widehat{\widehat\Omega}$ is of bounded type.
The action $\widehat G:\widehat{\widehat\Omega}$ is free, proper, and the orbits are totally real.
Hence the conclusion follows by Proposition \ref{Common}.
\end{proof}

\section{Action of a linear Lie group on a domain of bounded type}
Let $G\subset GL(n,\R)$ be a closed connected subgroup.
In order to apply Proposition \ref{Cover}, we need an action of $G$ on a domain of bounded type.
Let $G^c\subset GL(n,\C)$ be the complexification of $G$.
Let $\delta>0$ be small, and let
$$
E_\delta=E_\delta(G)=\{gp: g\in G, p\in G^c, |p-I|<\delta \}
\subset G^c
$$
be a $G$-invariant neighborhood of the identity matrix $I$ relative to the action $G:G^c$ by left translations.

\begin{prop}\label{Action}
  There exists a holomorphic free proper action $G:\Omega$ with totally real orbits on a domain $\Omega\subset\C^m$ of bounded type.
  The domain $\Omega$ is biholomorphically equivalent to the product $E_\delta\times\B^k$,
  here $\B^k\subset\C^k$ is the unit ball.
  The corresponding action $G:E_\delta\times\B^k$ consists of the natural action $G:G^c$ and the trivial action on $\B^k$, that is, $g\cdot(h,b)=(gh,b)$,
  here $g\in G$, $(h,b)\in E_\delta\times\B^k$, and $gh$ is matrix product.
\end{prop}

\begin{proof}
Let $\Sigma=\Sym(n,\C)$ be the set of all symmetric complex $n\times n$ matrices.
Let $\Sigma^+=\{z=x+iy\in \Sigma: x=\Re z>0\}$. Then $\Sigma^+$ is a Siegel domain, hence a domain of bounded type.
Let $Z=\Sigma\times \Sigma$ and $Z^+=\Sigma^+\times \Sigma^+$.
Then $Z^+\subset\C^m$, $m=n(n+1)$, is a domain of bounded type.

We note that the action $G:\Sigma^+$, $g\cdot z=gzg^T$, is proper. Then we define the actions $G^c:Z$ and $G:Z^+$ by
$$
g\cdot (z_1,z_2)=(gz_1g^T, gz_2g^T).
$$
Without loss of generality we assume $n$ is odd.
Then one can see that on an open dense set, the action $G:Z^+$ is free and proper.
Indeed, if $g\in G$ acts with a fixed point $z=(z_1,z_2)$, $z_\nu=x_\nu+iy_\nu$, $\nu=1,2$, then $g$ preserves all the four matrices $x_1$, $x_2$, $y_1$, $y_2$.
If they are generic, then $g=\pm I$. Since $n$ is odd and $G$ is connected, $-I\notin G$, hence $g=I$.

Fix a generic point $z_0=(z_{01},z_{02})\in Z^+$.
Define a map
$$
\Phi:G^c\to Z, \qquad
\Phi(h)= h\cdot z_0.
$$
One can see that the map $\Phi$ commutes with the action
$G:G^c$ by left translations and with the action $G:Z$
defined above.
We claim that for a small $\delta>0$, $\Phi|_{E_\delta}$ is an imbedding and $\Phi(E_\delta)\subset Z^+$. It also follows from the claim that $\Phi(G)$, the orbit of $z_0$
is totally real.

To prove the claim, it suffices to show that the tangent map
$\Phi_*(I):L(G^c)\to T_{z_0}Z^+$ is injective,
here $L(G^c)\subset \Mat(n,\C)$ is the Lie algebra of $G^c$. Let $\xi\in \ker(\Phi_*(I))$. Then
$$
\xi z_{0\nu}+z_{0\nu}\xi^T=0,\quad
\nu=1,2.
$$
We make a substitution
$\xi=QXQ^{-1}$, $z_{0\nu}=QA_\nu Q^T$ with
$\det Q\ne0$. Assuming $\det z_{01}\ne0$, we can choose $Q$ such that $A_1=I$. Putting $A=A_2$, we have the equations
$$
X+X^T=0,\quad
XA+AX^T=0.
$$
Then $XA=AX$. If $A$ is a diagonal matrix with distinct eigenvalues, then $X$ is also diagonal.
Since $X^T=-X$, we have $X=0$, hence $\xi=0$.
If follows that for a generic $A$, the conclusion $\xi=0$
still holds.
Hence, $\ker(\Phi_*(I))=\{0\}$, and $\Phi_*(I)$ is injective.
Thus the claim follows.

We now show that $G:\Omega\subset Z^+$,
$\Omega\cong E_\delta\times\B^k$.
Indeed, let $L\subset Z$ be a complex affine subspace of $Z$ through $z_0$ such that $L$ is complementary to $T_{z_0}\Phi(E_\delta)$ in $Z$.
Define $B=\{z\in L: |z-z_0|<\epsilon\}$,
where $\epsilon>0$ is small.
Then $B\cong\B^k$, $k=\dim Z-\dim G^c$.
Define $\Omega=\{h\cdot b: h\in E_\delta, b\in B\}$.
The domain $\Omega$ is of bounded type because
$\Omega\subset Z^+$.
Then the mapping
$(h,b)\mapsto h\cdot b$, $E_\delta\times B\to\Omega$, is biholomorphic.
\end{proof}

\begin{remark}
{\rm
Instead of using the action $G:Z^+$, it might seem
more natural to use the action by left translations $G:\Mat(n,\C)$.
Let $\Omega$ be a small G-invariant neighborhood of $I$ relative to this action. One can see that
$\Omega\subset M_\delta$, for some $\delta>0$. Here
$M_\delta=\{z=x+iy\in\Mat(n,\C): |y|<\delta |x| \}$.
However, we don't know whether $M_\delta$ is a domain of bounded type. Another option would be the action
$g\cdot z=gzg^T$ on the domain
$\{z=x+iy\in\Mat(n,\C): x+x^T>0 \}$,
which is a domain of bounded type.
For odd $n$, this action is free on a dense open set,
but generic orbits of $G$ are not totally real.
We finally note that the only reason we restrict to
the space $\Sigma$ of symmetric matrices is to reduce the dimension of $\Omega$.
}
\end{remark}

\section{Covering groups of compact groups}

We need some details on geometric structure of covering groups of compact Lie groups. The results of this section are useful when the covering group is not compact.

\begin{prop}
\label{Cover-Compact}
  Let $H$ and $\widehat H$ be connected Lie groups and let $H$ be compact.
  Let $f:\widehat H\to H$ be a covering homomorphism.
  Then there is an equivalent covering (for which we keep the same notation)
  and an admissible covering $F:\widehat{X}\to X$
  (see the definition before Proposition \ref{Cover})
  such that $H\subset X$, $\widehat H\subset\widehat X$,
  and $F|_{\widehat H}=f$.
\end{prop}

We first consider a special case.
Let $\T$ be the unit circle $\T=\{z\in\C:|z|=1\}$
with the natural projection
$\chi:\R\to \T$, $\chi(z)=e^{iz}$.

\begin{prop}
\label{Cover-free-kernel}
  Let $H$ and $\widehat H$ be connected Lie groups and let $H$ be compact.
  Let $f:\widehat H\to H$ be a covering homomorphism.
  Suppose $A=\ker f$ is a free abelian group of rank $n$.
  Then there exists a surjective homomorphism
  $\psi:H\to \T^n$ such that the covering $f:\widehat H\to H$
  is equivalent to $H\times_{\T^n}\R^n\to H$,
  the pullback by $\psi$ of the natural projection $\chi^n:\R^n\to\T^n$, and we have a commutative diagram:
  \[
  \begin{tikzcd}
  \widehat H \ar[r] \ar[d,"f"]
  & \R^n \ar[d, "\chi^n"]\\
  H \ar[r, "\psi"]
  & \T^n
  \end{tikzcd}
  \]
  Furthermore, the conclusion of Proposition \ref{Cover-Compact} holds.
\end{prop}

We need some preliminaries before the proof.

\begin{lemma}
\label{Fiberwise-product}
  Let $G$ be a group, and let $N_1$ and $N_2$ be normal subgroups of $G$.
  Let $H_\nu=G/N_\nu, \nu=1,2$, $H=G/(N_1N_2)$.
  Let $\pi_\nu:G\to H_\nu$ and $\phi_\nu:H_\nu\to H$ be the natural projections onto the quotients.
  Then $\chi:G\to H_1\times_H H_2$, $\chi=(\pi_1,\pi_2)$
  is a surjection. If $N_1\cap N_2=\{e\}$, then $\chi$ is an isomorphism.
\end{lemma}
We omit a simple proof.

\begin{lemma}
\label{Commutator-is-compact}
  Let $G$ and $H$ be connected Lie groups and let $H$ be compact.
  Let $f:G\to H$ be a covering homomorphism.
  Then the commutator $G'=[G,G]$ is compact.
\end{lemma}
\begin{proof}
  By a description of connected compact Lie groups,
  $H=(K\times\T^n)/B$. Here $K$ is a simply connected compact group, and $B\subset K\times\T^n$ is a discrete central subgroup.
  (See \cite{Procesi}, Chap. 10, Sec. 7.2, Theorem 4.)
  Then the universal covers
  $\widetilde G=\widetilde H=K\times\R^n$,
  and $(\widetilde G)'=K'\times 0$ is compact.
  Let $\pi:\widetilde G\to G$ be the covering homomorphism.
  Then $G'=\pi((\widetilde G)')$ is compact, as desired.
\end{proof}

\begin{lemma}
\label{Quotient-by-commutators}
  Let $G$ be a connected Lie group and let
  $A\subset G$ be a discrete central subgroup.
  Suppose $A$ is a free abelian group.
  Suppose $H=G/A$ is compact.
  Let $G_1=G/G'$ and $H_1=G/(G'A)$, here $G'$ is the commutator of $G$.
  Then the covering homomorphism $G\to H$ is equivalent to $H\times_{H_1}G_1\to H$.
\end{lemma}

This lemma reduces the study of a covering map $G\to H$
to a covering map $G_1\to H_1$ of abelian groups.

\begin{proof}
  By Lemma \ref{Commutator-is-compact}, $G'$ is compact.
  Since $A$ is free, $G'\cap A=\{e\}$. Then by Lemma \ref{Fiberwise-product} the conclusion follows.
\end{proof}

\begin{lemma}
\label{Cover-of-abelian}
  Let $G_1$ and $H_1$ be connected abelian Lie groups and let $H_1$ be compact.
  Let $f_1:G_1\to H_1$ be a covering homomorphism.
  Let $A=\ker f_1$ be a free abelian group of rank $n$.
  Then there exists an integer $m\ge0$ and isomorphisms
  $G_1\to K\times\R^n$ and $H_1\to K\times\T^n$
  such that the diagram
  $$
  \begin{tikzcd}
  G_1 \ar[r] \ar[d,"f_1"]
  & K\times\R^n \ar[d, "\id_K\times\chi^n"]\\
  H_1 \ar[r]
  & K\times\T^n
  \end{tikzcd}
  $$
  is commutative.
  Here $K=\T^m$ and
  $\chi^n:\R^n\to\T^n$ is the natural projection.
\end{lemma}

\begin{proof}
  Let $f:V\to G_1$ be the universal cover of $G_1$.
  Then $V$ is a real vector space.
  Let $C=\ker f$, $B=f^{-1}(A)$.
  Let $V_1=\Span_\R C$, and let $\gamma_1\subset C$ be a $\R$-basis of $V_1$.
  Since $B/C\simeq A$ is free, we can extend $\gamma_1$ to a $\Z$-basis $\gamma_1\cup\gamma_2$ of $B$.
  Let $V_2=\Span_\R\gamma_2$, and let $V_3\subset V$ be a
  complementary subspace to $V_1+V_2$.
  Then $K:=f(V_1)\simeq V_1/C$ is compact, $f_1|_K$ is an isomorphism, and $f|_{V_2+V_3}$ is an isomorphism.
  Then $f(V_2)\simeq\R^n$, $A\subset f(V_2)$, and
  $f_1(f(V_2))\simeq\T^n$.
  Since $V_3\simeq f_1(f(V_3))\subset H_1$ and $H_1$ is compact, we have $V_3=0$.
  Hence $G_1\simeq K\times\R^n$, $H_1\simeq K\times\T^n$,
  and the conclusion follows.
\end{proof}

\begin{proof}[Proof of Proposition \ref{Cover-free-kernel}]
  By Lemmas \ref{Quotient-by-commutators} and \ref{Cover-of-abelian},
  there are homomorphisms $\phi:H\to K$ and $\psi:H\to \T^n$
  such that the diagram
  $$
  \begin{tikzcd}
  \widehat H \ar[r] \ar[d,"f"]
  & K\times\R^n \ar[d, "\id_K\times\chi^n"]\\
  H \ar[r, "(\phi{,}\psi)"]
  & K\times\T^n
  \end{tikzcd}
  $$
  is commutative.
  Since $(\phi{,}\psi)$ is surjective, $\psi$ is surjective.
  Furthermore, $\widehat H$ can be identified with
  $H\times_{K\times\T^n}(K\times\R^n)$. Then we can redefine
  $$
  \widehat H=\{(h,k,x)\in H\times K\times\R^n: \phi(h)=k, \psi(h)=\chi^n(x)\},\quad
  f(h,k,x)=h.
  $$
  Since the component $k$ of $(h,k,x)\in \widehat H$ is completely determined by $h$, the factor $K$ is redundant, and we can further redefine
  $$
  \widehat H=\{(h,x)\in H\times\R^n:\psi(h)=\chi^n(x)\},\quad
  f(h,x)=h.
  $$
  Hence $\widehat H=H\times_{\T^n}\R^n$, and the first conclusion follows.

  We now show that $f$ extends to an admissible covering.
  The natural projection $\chi:\R\to \T$, $\chi(z)=e^{iz}$,
  extends as $\chi:B\to A$, here
  $B=\{z\in\C: |\Im z|<1\}$ is a strip, and
  $A=\chi(B)=\{z\in\C: e^{-1}<|z|<e\}$ is an annulus.

  Since $H$ is compact, it has a complexification $H^c$.
  We assume $H\subset \R^m$ and $H^c\subset \C^m$.
  Let $\epsilon>0$.
  Let $X=\{z\in H^c: \exists h\in H, |z-h|<\epsilon\}$
  be the $\epsilon$-thickening of $H$ in $H^c$. We choose
  $\epsilon$ small enough so that $X$ is diffeomorphic to $H\times\R^s$, $s=\dim H$.

  Shrinking $\epsilon$ if necessary, since $H$ is compact,
  we can assume that
  $\psi$ holomorphically extends to $X$. Keeping the notation
  $\psi$ for the extended map, we have $\psi:X\to (\C^*)^n$.
  Since $H$ is compact, we can assume $\psi(X)\subset A^n$.

  We define $F:\widehat X=X\times_{\psi(X)}B^n\to X$ as the pullback of $\chi^n:B^n\to A^n$ by $\psi$.
  Then $F|_H=f$. The elements of each preimage $F^{-1}(x)$
  are at least $2\pi$ units apart because $\chi$ has period $2\pi$. Finally, the bounded type condition is met because
  $X$ is bounded and $B^n$ is a domain of bounded type. Hence, $F$ is admissible, as desired.
\end{proof}

\begin{proof}[Proof of Proposition \ref{Cover-Compact}]
  Let $f:\widehat H\to H$ be a covering homomorphism. The abelian group $A=\ker f$ can be split into a direct sum $A=A_0+A_1$, in which $A_0$ is finite and $A_1$ is free. Accordingly,
  $f$ can be decomposed as $f=f_0\circ f_1$. Here
  $f_0:H_1\to H$, $f_1:\widehat H\to H_1$,
  $f_1(A_0)=\ker f_0$ is finite, and $A_1=\ker f_1$ is free.
  If $\widehat H$ is compact, then $A_1=0$,
  $\widehat H=H_1$, and $f_1=\id$.

  Since $H$ and $H_1$ are compact, they have complexifications $H^c$ and $H_1^c$, and the map $f_0$  holomorphically extends to a neighborhood of $H_1$ in $H^c_1$.
  Let $F_0$ denote the extension.
  Let $X$ be a small neighborhood of $H$ in $H^c$.
  We put $X_1=F^{-1}_0(X)$.
  Then the covering map $F_0:X_1\to X$ is admissible.

  By Proposition \ref{Cover-free-kernel}, if $X_1$ is a small enough neighborhood of $H_1$ in $H_1^c$, then there exists
  an admissible cover $F_1:\widehat X\to X_1$,
  where $\widehat H \subset \widehat X$,
  $F_1|_{\widehat H}=f_1$.
  Then $F=F_0\circ F_1: \widehat X\to X$ is admissible,
  as desired.
\end{proof}

\section{Polar decomposition and main results}
In order to apply Proposition \ref{Cover} we need a map $\psi$ defined in $\Omega$ and taking an orbit of $G$ to a compact set. By Proposition \ref{Action} it suffices to construct $\psi$ for $E_\delta$, a factor of $\Omega$. Recall
$$
E_\delta=\{gp: g\in G, p\in G^c, |p-I|<\delta \}
\subset G^c.
$$
We will construct $\psi$ using polar decomposition.
Recall that every real matrix $h\in GL(n,\R)$ has a unique polar decomposition $h=SQ$, where $S\in\Sym^+(n)$ is a positive definite symmetric matrix and $Q\in O(n,\R)$ is an orthogonal matrix.

Every complex matrix $h\in GL(n,\C)$ also has a polar decomposition $h=SQ$, where $S\in\Sym(n,\C)$ is a symmetric matrix and $Q\in O(n,\C)$ is an orthogonal matrix. Then $hh^T=SQQ^TS=S^2$, $S=\sqrt{hh^T}$, $Q=S^{-1}h=Sh^{T-1}$.

We will prove (Lemma \ref{Almost-Real}(d) and Proposition \ref{Eigenvalues}) that for $h\in E_\delta$, all eigenvalues $\lambda$ of the matrix $hh^T$ have $\Re\lambda>0$. In order to define the function
$\phi(h)=S=\sqrt{hh^T}$ we will use the branch of $\sqrt{\lambda}$ such that $\sqrt{\lambda}>0$ for $\lambda>0$. Then $\phi:E_\delta\to \Sym(n,\C)$ is a holomorphic function.
We define $\psi(h)=Q=\phi(h)h^{T-1}$.
Then $\psi:E_\delta\to O(n,\C)$ is a holomorphic function,
and for every $h\in E_\delta$, $\phi(h)\psi(h)=h$.
If $h\in G$, then $S=\phi(h)>0$, $Q=\psi(h)\in O(n,\R)$.

A subgroup $G\subset GL(n,\R)$ is called \emph{self-adjoint}
if $g\in G$ implies $g^T\in G$. Our main result is the following.
\begin{thm} \label{Main-Algebraic}
  Let $G$ be a real linear algebraic self-adjoint group. Let $G^0$ be the connected component of the identity of $G$.
  Let $\widehat{G}$ be a covering group for $G^0$.
  Then there is a strongly pseudoconvex domain $D\subset\C^m$ of bounded type such that $\Aut(D)\simeq\widehat{G}$.
\end{thm}

\begin{proof}
  We will use polar decomposition for a real linear algebraic self-adjoint group. (See \cite{Vinberg}, Chap. 1,
  Sec. 6.4.)
  Let $G\subset GL(n,\R)$.
  Let $P=G\cap \Sym^+(n)$, $K=G\cap O(n,\R)$.
  Then $G=PK$, the map $P\times K\to G$, $(p,k)\mapsto pk$, is a diffeomorphism, and $P$ is diffeomorphic to $\R^s$ for some $s$.

  Let $K^0$ be the connected component
  of the identity of $K$.
  Then $G^0=PK^0$, and we have
  $G^0\cap O(n,\R)=G^0\cap K=PK^0\cap K=K^0$.
  Since $K^0$ is connected, in fact
  $K^0=G^0\cap SO(n,\R)$. Using the above notation $\psi(h)$
  for the orthogonal component of $h$, we have
  $\psi(G^0)=K^0$, moreover, $\psi:G^0\to K^0$
  is a homotopy equivalence.

  Let $E_\delta=E_\delta(G^0)\subset (G^0)^c$ be as defined earlier. We will prove (Proposition \ref{Image-Orth-Decomp}) that for small $\delta>0$, $\psi(E_\delta)$ lies in a small neighborhood of $SO(n,\R)$. Since $\psi$ is holomorphic and $\psi(G^0)=K^0$, we have $\psi(E_\delta)\subset (K^0)^c$.
  Then in fact $\psi(E_\delta)$ is contained in a small neighborhood of $K^0$.

  We now apply Proposition \ref{Cover}, in which $G^0$ acts on the domain $\Omega$.
  Identifying $\Omega$ with $E_\delta\times\B^k$, we extend $\psi:E_\delta\to SO(n,\C)$ to $\Omega$ by putting $\psi(h,b)=\psi(h)$, where $(h,b)\in E_\delta\times\B^k$.
  Following the statement of Proposition \ref{Cover}, we define $M$ as the $G^0$-orbit of the point $z_0\in\Omega$
  and identify $M$ with $G^0$. Then $H=\psi(G^0)=K^0$.
  Shrinking $\Omega$ if necessary,
  we have $\psi(\Omega)\subset X$, where $X$ is a small neighborhood of $H=K^0$ in $(K^0)^c$.

  Since $\psi:G^0\to K^0$ is a homotopy equivalence, the covering $\widehat{G}\to G^0$ can be obtained by pulling
  back by $\psi$ some covering $f:\widehat{H}\to H$.
  By Proposition \ref{Cover-Compact} the latter can be extended to an admissible
  covering map $F:\widehat{X}\to X$.
  Theorem \ref{Main-Algebraic} now follows from Proposition \ref{Cover}.
\end{proof}

\begin{cor}\label{Main-Semisimple}
  Let $G$ be a real connected semisimple Lie group.
  Then there is a strongly pseudoconvex domain $D\subset\C^n$ of bounded type such that $\Aut(D)\simeq G$.
\end{cor}

\begin{proof}
  Let $\g$ be the Lie algebra of $G$.
  Then the group of automorphisms $\Aut(\g)$ is a linear algebraic group.
  Moreover, $\Aut(\g)$ is self-adjoint relative to a metric on $\g$ defined using the Cartan decomposition of $\g$.
  (See \cite{Vinberg}, Chap. 4, Sec. 3.2.)
  Consider the group of inner automorphisms
  $\Int(\g)=\Ad (G)\subset \Aut(\g)$.
  Then $\Int(\g)=\Aut(\g)^0$ is the connected component of $\Aut(\g)$ (see \cite{Vinberg}, Chap. 1, Sec. 3.3.),
  and $\Ad:G\to\Int(\g)$ is a covering map.
  Then the conclusion follows by Theorem \ref{Main-Algebraic}.
\end{proof}

In the rest of the paper we will prove
Propositions \ref{Eigenvalues} and \ref{Image-Orth-Decomp},
which we used in the proof of the main result.

\section{Complex numbers, vectors, and matrices close to real ones}
In the rest of the paper we deal with complex objects close to real ones. We have already introduced
$E_\delta=E_\delta(G)$, which is in a sense a small neighborhood of a real linear group $G$ in its complexification.
It suffices to consider $G=GL(n,\R)^0$,
the set of all real $n\times n$
matrices with positive determinant.
For $\delta>0$, we also define
\begin{align*}
& N_\delta^+=\{z=x+iy\in\C: |y|<\delta x \}, \\
& V_\delta=\{z=x+iy\in\C^n: |y|<\delta |x| \}, \\
& M_\delta=\{z=x+iy\in\Mat(n,\C): |y|<\delta |x| \},\\
& M_\delta^+=\{z=x+iy\in\Mat(n,\C):z^T=z, \pm y<\delta x\},
\end{align*}
here $x$ and $y$ are real.
The above inequality $\pm y<\delta x$ means that both
matrices $\delta x\pm y$ are positive definite.
If it does not cause confusion, we will write, say
$z\in V_\delta$, if in fact $z\in V_{\epsilon}$, $\epsilon=O(\delta)$.
We state simple properties of the objects defined above. We will use the bilinear inner product
$$
\<u,v\>=\sum u_j v_j.
$$

\vfill\eject

\begin{lemma}\label{Almost-Real}\
\begin{itemize}
  \item[(a)] Let $B$ be a complex symmetric matrix. Then $B\in M_\delta^+$ if and only if for every $x\in\R^n$ we have $\<Bx,x\>\in N_\delta^+$.
  \item[(b)] Let $h\in E_\delta$, $x\in\R^n$, $x\ne0$. Then $h^Tx\in V_\delta$.
  \item[(c)] Let $z\in V_\delta$, $z\ne0$. Then $\<z,z\>\in N_\delta^+$.
  \item[(d)] Let $B=hh^T$, $h\in E_\delta$. Then $B\in M_\delta^+$.
  \item[(e)] Let $A\in M_\delta$, $|B-I|<\delta$. Then $AB\in M_\delta$ and $BA\in M_\delta$.
  \item[(f)] Let $A\in M_\delta$, $B\in SO(n,\R)$. Then $AB\in M_\delta$ and $BA\in M_\delta$.
\end{itemize}
\end{lemma}
\begin{proof}
(d) Since $h\in E_\delta$, by (b) we have $h^Tx\in V_\delta$.
Then by (c) we have $\<Bx,x\>=\<h^Tx,h^Tx\>\in N_\delta^+$.
Then by (a) we have $B\in M_\delta^+$. The rest of the statements are obvious.
\end{proof}

\begin{remark}
{\rm
Despite the simplicity of the concept, one should use some caution. For instance, if $h\in E_\delta$ and $0\ne x\in\R^n$, then the inclusions $hx\in V_\delta$ or $hh^Tx\in V_\delta$ do not hold in general.
Likewise, if $h\in E_\delta$, then the inclusions $h^T h\in M_\delta^+$ or $(hh^T)^2\in M_\delta^+$ do not hold in general. Finally, $A,B\in M_\delta$ does not imply
$AB\in M_\delta$.
}
\end{remark}

We estimate the eigenvalues of matrices in $M_\delta^+$.
\begin{prop}\label{Eigenvalues}
  Let $B\in M_\delta^+$ with small $\delta>0$, and let $\lambda$ be an eigenvalue of $B$.
  Then $\lambda\in N_\epsilon^+$, $\epsilon=O(\delta)$. In particular,
  $\Re \lambda>0$.
\end{prop}

For our application, we only need $\Re \lambda>0$,
but the stronger conclusion $\lambda\in N_\epsilon^+$ holds.

\begin{proof}
Let $B=B_0+iB_1$ with real $B_0$ and $B_1$.
Then $\pm B_1<\delta B_0$.
By replacing $\delta B_0$ by$B_0$, we can temporarily put $\delta=1$.
There exists $U\in SO(n,\R)$ such that $B_0=U\Lambda U^T$. Here $\Lambda=\Diag(\lambda_1,\ldots,\lambda_n)>0$ is a diagonal matrix containing the eigenvalues of $B_0$.

We put $B_1=UA U^T$, $A=(a_{ij})$. Then $\Lambda\pm A>0$.
For the diagonal elements of the latter, we have
$\lambda_j\pm a_{jj}>0$, hence $|a_{jj}|<\lambda_j$.
For a pair of unequal indices, say 1 and 2, we have
$$
\begin{pmatrix}
  \lambda_1+a_{11} & a_{12} \\
  a_{21} & \lambda_2+a_{22}
\end{pmatrix} >0.
$$
Then the determinant of the above matrix is
$(\lambda_1+a_{11})(\lambda_2+a_{22})-a_{12}^2>0$.
Since $|a_{jj}|<\lambda_j$, we have
$|a_{12}|<2\sqrt{\lambda_1 \lambda_2}$. Hence for all indices we have
$$
|a_{ij}|<2\sqrt{\lambda_i \lambda_j}.
$$
Then
$B=UB'U^T$,
$B'=\Lambda+iA=\Lambda^{1/2}(I+iK)\Lambda^{1/2}$.
Here $K=\Lambda^{-1/2}A\Lambda^{-1/2}=(K_{ij})$,
$K_{ij}=a_{ij}\lambda_i^{-1/2}\lambda_j^{-1/2}$.
Then $|K_{ij}|<2$. For arbitrary $\delta>0$, we have
$|K_{ij}|<2\delta$, $|K|<\delta_1:=2n\delta$.

Let $\lambda$ be an eigenvalue of $B$. Then $\lambda$ is an eigenvalue of $B'$ as well. Then there is $z\ne0$ such that $B'z=\lambda z$.
Put $w=\Lambda^{1/2}z$. Then we have
$(I+iK)w=\lambda\Lambda^{-1}w$.
Put $\Gamma=I-\lambda\Lambda^{-1}=\Diag(\gamma_1,\ldots,\gamma_n)$,
$\gamma_j=1-\lambda\lambda_j^{-1}$.
Then
$$
\Gamma w=-iKw, \quad
|\Gamma w|<\delta_1|w|.
$$
Then there is $j$ such that
$|\gamma_j|=|1-\lambda\lambda_j^{-1}|<\delta_1$.
Then $\lambda\in\C$ lies on a ray from 0 passing through the point $\lambda\lambda_j^{-1}$ in the disc
$\{z\in\C: |z-1|<\delta_1\}$, that is,
$\lambda$ lies in the open sector in $\C$ spanned by that disc.
Then $\lambda\in N_\epsilon^+$,
$\epsilon=\delta_1/(1-\delta_1^2)^{-1/2}=O(\delta)$.
\end{proof}

\section{Square roots of complex matrices close to positive definite ones}

Let $B\in M_\delta^+$. Then by Proposition \ref{Eigenvalues},
for every eigenvalue $\lambda$ of $B$, we have $\Re\lambda>0$. The function $\sqrt{\lambda}$ has a branch in the domain $\{\lambda\in\C: \Re\lambda>0\}$ with $\sqrt{\lambda}>0$ if $\lambda>0$. Using that branch of $\sqrt{\lambda}$, we define $S=\sqrt{B}$. We will call this $S$ the \emph{principal} square root of $B$.

\begin{lemma}\label{ReS>0}
  Let $B\in M_\delta^+$. Let $S=\sqrt{B}$ be the principal square root of $B$. Then $\Re S>0$.
\end{lemma}
\begin{proof}
Let $S=S_0+iS_1$ with real $S_0$ and $S_1$.
Then $Re B=S_0^2-S_1^2>0$, and $S_0^2>0$.
Then $\det S_0^2\ne0$, whence $\det S_0\ne0$.

Note that $M_\delta^+$ is a convex cone, a connected set.
The function $\phi(B)=\Re\sqrt{B}$, where $\sqrt{B}$ is the principal square root of $B$, is a continuous function
on $M_\delta^+$, and $\det\phi(B)\ne0$.
Note for $I\in M_\delta^+$, we have  $\phi(I)=I>0$. If $\phi(B_0)$ has a negative eigenvalue at some $B_0\in M_\delta^+$, then $\det\phi(B)=0$ for some $B\in M_\delta^+$, which can not occur. Hence $\Re S=\phi(B)>0$ for all $B\in M_\delta^+$.
\end{proof}

\begin{prop}\label{Description-Sqrt(B)}
  Let $B\in M_\delta^+$. Let $S=\sqrt{B}$ be the principal square root of $B$. Then $S\in M_\epsilon^+$, $\epsilon=O(\delta)$.
  Furthermore, $S=U(I+iK)\Lambda U^T$. Here $U\in SO(n,\R)$,
  $\Lambda>0$ is diagonal, and $K=O(\delta)$.
\end{prop}

For our application, we only need the formula for $\sqrt{B}$, which implies the conclusion $\sqrt{B}\in M_\epsilon^+$.
Curiously, the statements
$B^2\in M_\epsilon^+$ and even $\Re (B^2)\ge0$ are false.

\begin{proof}
As above, $S=S_0+iS_1$ with real $S_0$ and $S_1$.
Then $B=S^2=S_0^2-S_1^2+i(S_0S_1+S_1S_0)$.
Since $B\in M_\delta^+$, we have
$\delta(S_0^2-S_1^2)\pm(S_0S_1+S_1S_0)>0$.
In particular, $\delta S_0^2\pm(S_0S_1+S_1S_0)>0$.
By replacing $S_1$ by $\delta^{-1}S_1$, we can temporarily
put $\delta=1$.

There exists $U\in SO(n,\R)$ such that $S_0=U\Lambda U^T$. Here $\Lambda=\Diag(\lambda_1,\ldots,\lambda_n)$
is a diagonal matrix containing the eigenvalues of $S_0$.
By Lemma \ref{ReS>0}, we have $S_0>0$, whence $\Lambda>0$.

We put $S_1=UA U^T$, $A=(a_{ij})$.
Then $\Lambda^2\pm (\Lambda A+A\Lambda)>0$.
Note $(\Lambda A+A\Lambda)_{ij}=(\lambda_i+\lambda_j)a_{ij}$.
For diagonal elements of
$\Lambda^2\pm (\Lambda A+A\Lambda)$, we have
$\lambda_j^2\pm 2\lambda_j a_{jj}>0$, hence $|a_{jj}|<\lambda_j/2$.
For a pair of unequal indices, say 1 and 2, we have
$$
\begin{pmatrix}
\lambda_1^2+2\lambda_1 a_{11}& (\lambda_1+\lambda_2)a_{12}\\
(\lambda_1+\lambda_2)a_{21} & \lambda_2^2+2\lambda_2 a_{22}
\end{pmatrix} >0.
$$
Then the determinant of the above matrix is
$(\lambda_1^2+2\lambda_1 a_{11})
(\lambda_2^2+2\lambda_2 a_{22})
-(\lambda_1+\lambda_2)^2 a_{12}^2>0$.
Since $|a_{jj}|<\lambda_j/2$, we have
$|a_{12}|<2{\lambda_1 \lambda_2}(\lambda_1 +\lambda_2)^{-1}
\le 2\min(\lambda_1, \lambda_2)$. Hence, for all indices and arbitrary $\delta>0$, we have
$$
|a_{ij}|<2\delta\min(\lambda_i, \lambda_j).
$$
For every $x\in\R^n$, $y=U^Tx$, we have
\begin{align*}
|\<S_1x,x\>| & =|\<Ay,y\>|
<2\delta\sum_{i,j}\min(\lambda_i, \lambda_j)|y_iy_j|
\le\delta\sum_{i,j}\min(\lambda_i, \lambda_j)(y_i^2+y_j^2)
\\
& \le 2\delta\sum_{i,j}\lambda_jy_j^2
=2n\delta\sum_{j}\lambda_jy_j^2
=2n\delta\<\Lambda y,y\>=2n\delta\<S_0 x,x\>.
\end{align*}
Then $\<Sx,x\>\in N_\epsilon^+$, $\epsilon=2n\delta=O(\delta)$.
By Lemma \ref{Almost-Real}(a), $S\in M_\epsilon^+$
as desired. Furthermore,
$$
S=U(\Lambda+iA)U^T=U(I+iK)\Lambda U^T.
$$
Here $K=A\Lambda^{-1}$, $K_{ij}=a_{ij}\lambda_j^{-1}$,
$|K_{ij}|<2\delta \min(\lambda_i, \lambda_j)\lambda_j^{-1}
\le 2\delta$. Hence $K=O(\delta)$ as desired.
\end{proof}

\section{Complex orthogonal matrices close to real ones}
\begin{lemma}\label{Orth-close-to-real}
Let $Q\in SO(n,\C)\cap M_\delta$ with small $\delta>0$. Then
there exists $U\in SO(n,\R)$ such that $|Q-U|<\epsilon$,
$\epsilon=O(\delta)$.
\end{lemma}
\begin{proof}
We put $Q=Q_0+iQ_1$ with real $Q_0$ and $Q_1$.
For every $Q\in SO(n,\C)$, we have
$|Q_0|^2=|Q_1|^2+1$.
Indeed, for every $x\in\R^n$ we have $\<Qx,Qx\>=\<x,x\>$.
Taking the real parts, we have
$|Q_0x|^2-|Q_1x|^2=|x|^2$.
Then
$|Q_0|^2=\max_{|x|=1}|Q_0x|^2
=\max_{|x|=1}|Q_1x|^2+1=|Q_1|^2+1$.

If $Q\in SO(n,\C)\cap M_\delta$, then
$|Q_0|^2=|Q_1|^2+1 < \delta^2 |Q_0|^2+1$.
Then $|Q_0|< (1-\delta^2)^{-1/2}$,
$|Q_1|< \delta(1-\delta^2)^{-1/2}=O(\delta)$, and
$|Q|=1+O(\delta)$.

Let $Q_0=SU$ be the polar decomposition of the real matrix
$Q_0$. Here $S=S^T>0$, $U\in SO(n,\R)$,
$S=\sqrt{B}$, $B=Q_0 Q_0^T$.
The matrix $B$ is close to the identity. Indeed,
$B=Q_0 Q_0^T=(Q-iQ_1)(Q-iQ_1)^T=I+O(\delta)$.
Then $S=\sqrt{B}=I+O(\delta)$ and $S^{-1}=I+O(\delta)$.
Finally $U=S^{-1}Q_0=Q_0+O(\delta)=Q+O(\delta)$,
hence $Q-U=O(\delta)$ as desired.
\end{proof}

We recall that $\psi:E_\delta\to SO(n,\C)$ is the orthogonal part of a complex matrix in polar decomposition.
The following proposition plays a crucial role in the proof of the main result.
\begin{prop}\label{Image-Orth-Decomp}
For small $\delta>0$, the image $\psi(E_\delta)$ lies in the $\epsilon$-thickening of $SO(n,\R)$, $\epsilon=O(\delta)$.
\end{prop}
\begin{proof}
Let $h\in E_\delta$, $h=gp$,
$g\in\Mat(n,\R)$, $\det g>0$,
$p\in\Mat(n,\C)$, $|p-I|<\delta$.
Then $p^{T-1}=I+O(\delta)$.

By definition, $\psi(h)=Q$ is the orthogonal part of the polar decomposition $h=SQ$.
Here $S=\sqrt{B}$ is the principal square root of $B=hh^T$.
Recall that by Lemma \ref{Almost-Real}(d), $B\in M_\delta^+$.
By Proposition \ref{Eigenvalues}, the eigenvalues of $B$ have positive real parts, hence $S$ is well defined.
Then $\psi(h)=Q=S^{-1}h=Sh^{T-1}=Sg^{T-1}p^{T-1}$.

By Proposition \ref{Description-Sqrt(B)} we have
$S=U(I+iK)\Lambda U^T$, here $U\in SO(n,\R)$, $K=O(\delta)$,
$\Lambda>0$ is diagonal. Then
$$
Q=U(I+iK)\Lambda U^T g^{T-1}p^{T-1}.
$$
We note that $Q_1:=\Lambda U^T g^{T-1}\in M_\delta$
because it is real.
By Lemma \ref{Almost-Real}(e), $Q_2:=(I+iK)Q_1\in M_\delta$
because $K=O(\delta)$.
By Lemma \ref{Almost-Real}(f), $Q_3:=UQ_2\in M_\delta$
because $U\in SO(n,\R)$.
Again by Lemma \ref{Almost-Real}(e), $Q=Q_3 p^{T-1}\in M_\delta$
because $p^{T-1}=I+O(\delta)$.
Finally, by Lemma \ref{Orth-close-to-real},
since $Q\in SO(n,\C)\cap M_\delta$, the matrix $\psi(h)=Q$
is $\epsilon$-close to a matrix in $SO(n,\R)$,
$\epsilon=O(\delta)$, as desired.
\end{proof}

\end{document}